\documentclass{amsart}

\usepackage[T1]{fontenc}
\usepackage{lmodern}
\usepackage{comment}

\usepackage[francais,english]{babel}

\usepackage{latexsym}
\usepackage{amssymb}

\usepackage[centering,textheight=50.5pc,textwidth=30pc]{geometry}
\usepackage[pdfstartview=XYZ]{hyperref}
\theoremstyle{plain}
\newtheorem{theorem}{Theorem}[section]

\newtheorem{lemma}[theorem]{Lemma}

\theoremstyle{definition}
\newtheorem{definition}[theorem]{Definition}

\theoremstyle{remark}

\theoremstyle{conjecture}
\newtheorem{conjecture}[theorem]{Conjecture}

\def\N{\mathbb{N}}

\def\C{\mathbb{C}}
\def\Q{\mathbb{Q}}

\title[Orders of primes and order of a circulant Hadamard matrix]
{If $p^a \vert \vert n$  where $n >4$ is the order of a Circulant Hadamard matrix,
then the order of $p$ modulo $n/p^a$ is odd}

\date{\today}

\author[Luis. H. Gallardo]{Luis H. Gallardo}

\address{University Of Brest,
Mathematics\\
6, Av. Le Gorgeu\\
C.S. 93837\\
29238 Brest Cedex 3, France.}

\email{Luis.Gallardo@univ-brest.fr}



\begin{document}


\begin{abstract}
We proved recently (see \cite{lhgarasu}) the result on the title for odd prime divisors of such an $n.$ The result implies
for many $n's$, more precisely,
for an infinity  of $n$'s with an arbitrary fixed number of prime divisors,  the inexistence
of circulant Hadamard matrices, and the inexistence of Barker sequences of length $n >13$.
The proof used a result of Arasu. It turns out that
there is another, shorter proof, of the more general result that includes the prime $p=2.$ This new proof is based on a result 
of Brock (see \cite[Theorem 3.1]{brock}), and besides that, requires just the definition of the Fourier transform. I noticed Brock's
 result in a preprint (see \cite{winterhofetal}) of Winterhof et al.
where it is used to study the inexistence of related Butson-Hadamard matrices. 
\end{abstract}

\maketitle



\section{Introduction}

A complex matrix $H$ of order $n$ is \emph{complex Hadamard} if $HH^{*} = nI$, where $I$ is the identity matrix of order $n$,
and if every entry  of $H/\sqrt{n}$ 
is in the complex unit circle. Here, the $*$ means transpose and conjugate. 
When such $H$ has real entries, so that $H$ is a $\{-1,1\}$- matrix, $H$ is called \emph{Hadamard}.  If $H$ is Hadamard
and circulant, say $H=circ(h_1,\ldots,h_n)$, that means that the $i$-th row $H_i$ of $H$ is given by 
$H_i = [h_{1-i+1},\ldots, h_{n-i+1}]$, the subscripts being taken modulo $n,$ for example
$H_2 =[h_n,h_1,h_2, \ldots,h_{n-1}].$ A long standing conjecture of Ryser (see \cite[pp. 134]{ryser}) 
is:

\begin{conjecture}
\label{mainryser}
Let $n \geq 4.$  If $H$ is a circulant Hadamard matrix of order $n$, then $n=4.$
\end{conjecture}

Details about previous results on the conjecture and a short sample of recent related papers are in
\cite{Leung}, \cite{mossinghoff}, \cite{Craigen}, \cite{brualdi}, \cite{EGR}, \cite{lhg}, \cite{lhgarasu}.

The latter paper \cite{lhgarasu}, based in \cite{arasu}, contains a proof of a practical criterion to reject many values
of  an integer $n$ 
with $n>4,$ as possible orders
of a circulant Hadamard matrix $H.$

The object of the present paper is to prove the following generalization of these result:

\begin{theorem}
\label{allpes}
Let $H$ be a circulant Hadamard matrix of order $n = p_1^{2a_1} \cdots p_t^{2a_t}$ where $p_1=2,$ $a_1=1$
and $p_2, \ldots,p_t$ are all the distinct odd prime factors of $n.$  Set for all $j=1, \ldots, t$ $m_j := n/p_j^{2
a_j}.$
Then,
$$
o_m(p)
$$
is odd,  for any prime number $p$ that divides $n$, where $m :=m_j$ if $p :=p_j.$
\end{theorem}

The main result used in the proof is \cite[Theorem 3.1]{brock} (see also Lemma \ref{brockmain} below). We discovered it
in the recent preprint \cite{winterhofetal} of A. Winterhof, O. Yayla and V. Ziegler, where it is used to prove inexistence
results on some Butson-Hadamard matrices (complex Hadamard  matrices with all its entries $k$-th roots of $1$ for some $k$).

\section{Some tools}

Brock's result (with some rewording) is the following.

\begin{lemma}
\label{brockmain}
Let $n = {\vert a \vert }^2 \in \N$ for some $a \in Q(w_{n_1}),$ where $w_{n_1} = exp(2 i \pi/n_1)$,
then for every prime $p \nmid n_1$
such that $p \mid n$ the order $o_{n_1}(p)$ of $p$ modulo $n_1$ is odd.
\end{lemma}

\begin{proof}
See \cite[Theorem 3.1]{brock}.
\end{proof}

The following is well known.

\begin{lemma}
\label{eigensH}
A circulant Hadamard matrix of order $n$ exists if and only if $\vert \lambda \vert = \sqrt{n}$ for each eigenvalue $\lambda$ of $H.$
\end{lemma}

We recall the definition of the Fourier matrix,  and in the lemma below the fundamental property of the Fourier transform
for which we may check \cite[pp .32--35]{davis}.

\begin{definition}
\label{fourier}
Let $n$ be a positive integer. Let $\omega := exp(2 i \pi/n).$ Then,
the Fourier matrix $F$ is defined by the equality
$$
\sqrt{n}F^{*} = Vandermonde([1,\omega,\omega^2,\ldots,\omega^{n-1}]) = (\omega^{(i-1)(j-1)}).
$$ 
\end{definition}

\begin{lemma}
\label{fourierT}
Let $n$ be a positive integer.
 If $A = circ(a_1, \ldots,a_n)$ is a circulant matrix, of order $n$, with complex entries $a_j \in \C$ then its eigenvalues,
$b_j$, $j=1, \ldots, n$
in some order, are given by
$$
\sqrt{n} \cdot [\overline{a_1},\ldots,\overline{a_n}] \cdot F = [\overline{b_1},\ldots,\overline{b_n}]
$$
where the $\overline{\cdot}$ denotes complex conjugation.
\end{lemma}

We are now ready to prove our main result in next section.

\section{Proof of Theorem \ref{allpes}}

Set $w_n := exp(2i \pi/n).$ Moreover,
for any given $g \in \{1, \ldots,n\},$  put $n_1 := m_g.$ Set $w_{n_1} := exp(2 i \pi/n_1).$

Let $H := circ(h_1, \ldots,h_n).$ Let $R(x) := h_1+h_2x + \cdots+h_nx^{n-1}$ be the \emph{representer} polynomial
of $H.$ By Lemma \ref{fourierT}, the eigenvalues of $H$ are the $b_s := R(w^{s-1})$ for all $s=1,\ldots,n.$ 
Define the index $j \in \{1,\ldots,n\}$ by
\begin{equation}
\label{lej}
p_g^{2 a_g} = j -1.
\end{equation}
Observe that 
$$
n = (2h)^2
$$
where 
$h :=2\cdot p_2^{a_2} \cdots p_t^{a_t}.$
Now, by Lemma
\ref{eigensH} one has also
\begin{equation}
\label{pro}
n = b_j \cdot \overline{b_j} = {\vert b_j \vert}^2.
\end{equation}

Observe that by definition of $w_{n_1}$ and by (\ref{lej}) one has
$$
w_{n_1} = w_n^{n/n_1} = w_n^{j-1}.
$$

Thus, 
\begin{equation}
\label{cond1}
b_j = R(w_n^{j-1}) = R(w_{n_1}) \in \Q(w_{n_1}).
\end{equation}

Set $p := p_g.$  Clearly $p \nmid n_1$ and $p \mid n.$
Therefore, applying Lemma \ref{brockmain} to  $p$ one gets that
$$
o_{n_1}(p)
$$
is odd.
This proves the theorem.










\end{document}